\documentclass[12pt]{article}
\usepackage{amssymb}
\usepackage{amsmath}
\usepackage{amsthm}
\usepackage{color}
\usepackage{comment}
\usepackage{geometry}		
\usepackage{graphicx}

\makeatletter
	\@addtoreset{equation}{section}
\makeatother

\let\originalleft\left
\let\originalright\right
\renewcommand{\left}{\mathopen{}\mathclose\bgroup\originalleft}
\renewcommand{\right}{\aftergroup\egroup\originalright}

\begin{document}

\newcommand\cF{\mathcal{F}}
\newcommand\cG{\mathcal{G}}
\newcommand\cX{\mathcal{X}}
\newcommand\cY{\mathcal{Y}}
\newcommand\cZ{\mathcal{Z}}

\newcommand{\dfem}[1]{{\bf #1}}

\newtheorem{theorem}{Theorem}[section]
\newtheorem{corollary}[theorem]{Corollary}
\newtheorem{lemma}[theorem]{Lemma}
\newtheorem{proposition}[theorem]{Proposition}

\theoremstyle{definition}
\newtheorem{definition}{Definition}[section]
\newtheorem{example}[definition]{Example}

\theoremstyle{remark}
\newtheorem{remark}{Remark}[section]




\title{
The necessity of the sausage-string structure for mode-locking regions of piecewise-linear maps.
}
\author{
D.J.W.~Simpson\\\\
School of Mathematical and Computational Sciences\\
Massey University\\
Palmerston North, 4410\\
New Zealand
}
\maketitle


\begin{abstract}

Piecewise-smooth maps are used as discrete-time models of dynamical systems
whose evolution is governed by different equations under different conditions (e.g.~switched control systems).
By assigning a symbol to each region of phase space where the map is smooth,
any period-$p$ solution of the map can be associated to an itinerary of $p$ symbols.
As parameters of the map are varied,
changes to this itinerary occur at border-collision bifurcations (BCBs) where 
one point of the periodic solution collides with a region boundary.
It is well known that BCBs conform broadly to two cases:
{\em persistence}, where the symbolic itinerary of a periodic solution changes by one symbol,
and a {\em nonsmooth-fold}, where two solutions differing by one symbol collide and annihilate.
This paper derives new properties of periodic solutions of piecewise-linear continuous maps on $\mathbb{R}^n$
to show that under mild conditions
BCBs of mode-locked solutions on invariant circles must be nonsmooth-folds.
This explains why Arnold tongues of piecewise-linear maps exhibit a sausage-string structure
whereby changes to symbolic itineraries occur at codimension-two pinch points instead of codimension-one persistence-type BCBs.
But the main result is based on the combinatorical properties of the itineraries,
so the impossibility of persistence-type BCBs also holds when the periodic solution is unstable or there is no invariant circle.

\end{abstract}

\section{Introduction}
\label{sec:intro}

Many natural phenomena involve one oscillatory system driving another.
Examples include tidal currents driven by the moon's rotation \cite{Ge19},
circadian rhythms driven by the spin of the Earth \cite{En80},
and bridge sway caused by the movement of people \cite{Ga15}.
Often the motion of the driven system becomes {\em mode-locked} to that of the driving system,
meaning both systems are periodic with the same period or periods admitting a rational ratio.
In two-parameter bifurcation diagrams, mode-locking occurs in {\em Arnold tongues}.
Fig.~\ref{fig:nopeSmoothTongues} shows three Arnold tongues for the two-dimensional quadratic map
\begin{equation}
\begin{bmatrix} x_1 \\ x_2 \end{bmatrix} \mapsto
\begin{bmatrix} \alpha x_1 + x_2 + 1 \\ \beta x_1 - \frac{1}{2} x_1^2 \end{bmatrix},
\label{eq:exampleMapSmooth}
\end{equation}
where $\alpha$ and $\beta$ are parameters.
In this abstract setting the presence of a driving system is obscured but this is not important for our purposes.
Arnold tongues have been described in mathematical models of diverse applications,
such as lasers \cite{JaJi19}, neurons \cite{CoBr99}, and cardiac dynamics \cite{GlGo87}.

\begin{figure}[b!]
\begin{center}
\includegraphics[width=17cm]{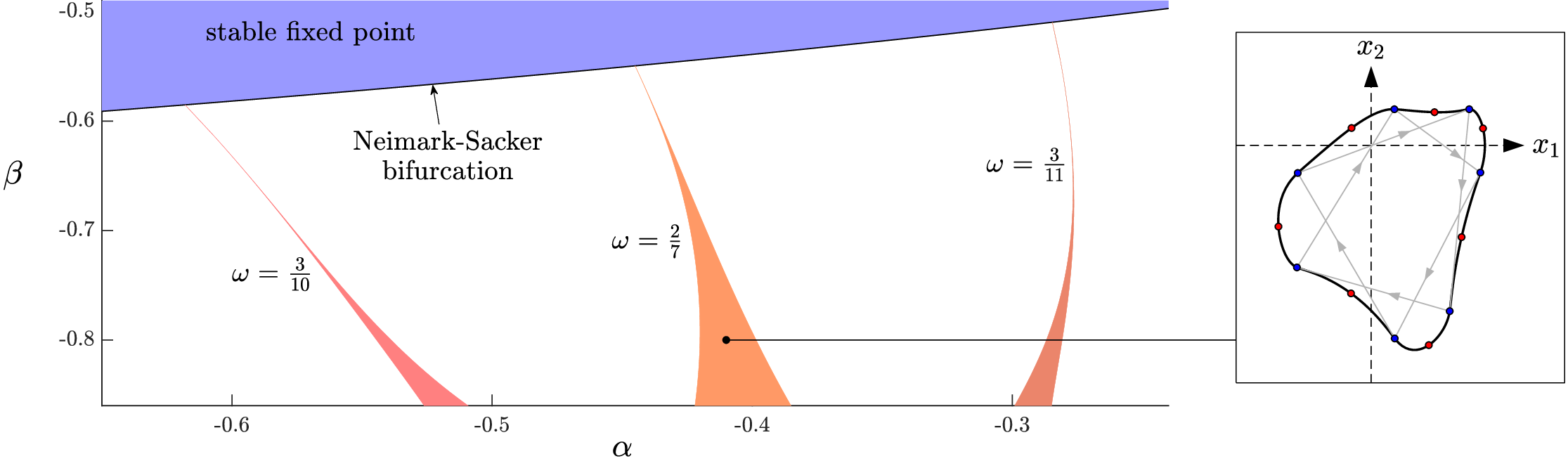}
\caption{
A two-parameter bifurcation diagram of the smooth map \eqref{eq:exampleMapSmooth}.
This map has a stable fixed point throughout the blue region which
is bounded by a curve of Neimark-Sacker bifurcations.
On this curve the stability multipliers associated with the fixed point are ${\rm e}^{2 \pi {\rm i} \omega}$
where the value of $\omega \in (0,1)$ values continuously along the curve.
Arnold tongues emanate from points on the curve at which $\omega$ is rational.
Three tongues are shown (computed numerically by continuing their boundaries
where the associated stable periodic solution undergoes a saddle-node bifurcation).
For a point in the middle tongue we provide a phase portrait
showing the stable (blue) and saddle (red) periodic solutions on the attracting invariant circle (black).
The grey lines show how each point of the stable periodic solution maps to the next point
giving a rotation number of $\omega = \frac{2}{7}$.
\label{fig:nopeSmoothTongues}
} 
\end{center}
\end{figure}

Below the Neimark-Sacker bifurcation curve in Fig.~\ref{fig:nopeSmoothTongues} the map has an attracting invariant circle.
Arnold tongues are where this circle contains stable and saddle solutions of some fixed period $p$ and rotation number $\omega = \frac{m}{p}$;
elsewhere the dynamics on the circle is quasi-periodic.
Additional bifurcations occur at more negative values of $\beta$,
such as period-doubling and the destruction of the invariant circle \cite{Ar88,ArCh82,MePa93},
but these are not our concern here.    
The main point is that in Fig.~\ref{fig:nopeSmoothTongues}
and for models where the equations of motion are smooth,
each tongue is usually a connected set.
This contrasts Arnold tongues of piecewise-linear maps that usually display a sausage-string structure.
Fig.~\ref{fig:nopePWLTongues} illustrates this for the map
\begin{equation}
\begin{bmatrix} x_1 \\ x_2 \end{bmatrix} \mapsto
\begin{bmatrix} \alpha x_1 + x_2 + 1 \\ \beta x_1 - \frac{1}{2} |x_1| \end{bmatrix},
\label{eq:exampleMapPWS}
\end{equation}
obtained from the previous example by replacing $x_1^2$ with $|x_1|$.
Each tongue has zero width at pinch points termed {\em shrinking points} and is a disjoint union of `sausages'.
This has been described for diverse applications including
power converters \cite{ZhMo03}, trade cycles \cite{GaGa03},
and mechanical oscillators subject to dry friction \cite{SzOs09}.
Note that more tongues have been shown in Fig.~\ref{fig:nopePWLTongues} than Fig.~\ref{fig:nopeSmoothTongues}
simply because in the piecewise-linear setting it is considerably easier to compute them to a reasonable degree of precision.

\begin{figure}[b!]
\begin{center}
\includegraphics[width=17cm]{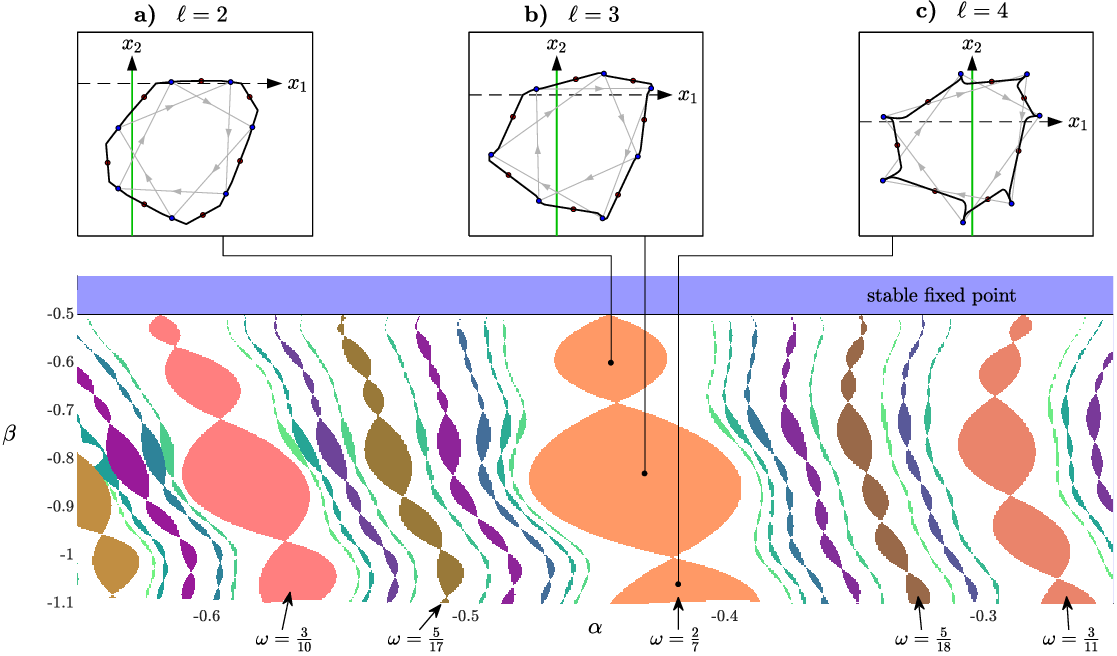}
\caption{
A two-parameter bifurcation diagram of the piecewise-linear map \eqref{eq:exampleMapPWS}.
This map has a stable fixed point throughout the blue region bounded by the line $\beta = -0.5$.
This line is a piecewise-linear analogue of a curve of Neimark-Sacker bifurcations \cite{SiMe08b,SuGa08}
where the fixed point has stability multipliers ${\rm e}^{2 \pi {\rm i} \omega}$
with $\omega = \frac{1}{2 \pi} \cos^{-1} \left( \frac{\alpha}{2} \right)$.
Arnold tongues emanate from points where $\omega$ is rational
and these have been computed numerically up to period $50$
(by solving for the periodic solutions algebraically).
Phase portraits for three points in the $\omega = \frac{2}{7}$ tongue are also shown.
\label{fig:nopePWLTongues}
} 
\end{center}
\end{figure}

This paper presents new results for $n$-dimensional continuous maps of the form
\begin{equation}
x \mapsto \begin{cases}
A_L x + b, & x_1 \le 0, \\
A_R x + b, & x_1 \ge 0,
\end{cases}
\label{eq:pwl}
\end{equation}
where $x = (x_1,x_2,\ldots,x_n)$,
$A_L$ and $A_R$ are $n \times n$ matrices differing only in their first columns (for continuity),
and $b \in \mathbb{R}^n$.
The previous example \eqref{eq:exampleMapPWS} is a two-dimensional subfamily of \eqref{eq:pwl}.
As in the smooth setting, Arnold tongues
correspond to stable periodic solutions on invariant circles with a given rotation number,
but for the piecewise-linear map \eqref{eq:pwl}
it is helpful to also consider the number of points $\ell$ that the solution has in the left half-plane $x_1 < 0$.
In general as we move from one sausage to the next the value of $\ell$ changes by one,
as shown in Fig.~\ref{fig:nopePWLTongues} for the $\omega = \frac{2}{7}$ tongue.
These changes occur at codimension-two points,
and there is substantial mathematical theory explaining {\em how} these occur \cite{SiMe09,Si17c,Si18e},
but the theory does not explain {\em why} the changes must occur in this way.
For instance, why can't increments in $\ell$ occur in a codimension-one fashion, as suggested in Fig.~\ref{fig:nopePartitionedTongue}
where one point of the periodic solution simply crosses the switching manifold?
Certainly such bifurcations are possible when there is no attracting invariant circle.
Fig.~\ref{fig:nopeTonguesSpecial10} shows an example using
\begin{equation}
A_L = \begin{bmatrix} \tau_L & 1 \\ -\delta_L & 0 \end{bmatrix}, \quad
A_R = \begin{bmatrix} \tau_R & 1 \\ -\delta_R & 0 \end{bmatrix}, \quad
b = \begin{bmatrix} 1 \\ 0 \end{bmatrix}.
\label{eq:exampleNonrotational}
\end{equation}
The bifurcation is a border-collision bifurcation (BCB) of persistence type
whereby one point of a stable period-$9$ solution crosses $x_1 = 0$
and the number of points the solution has in $x_1 < 0$ changes from $\ell = 3$ to $\ell = 4$.

\begin{figure}[b!]
\begin{center}
\includegraphics[width=5cm]{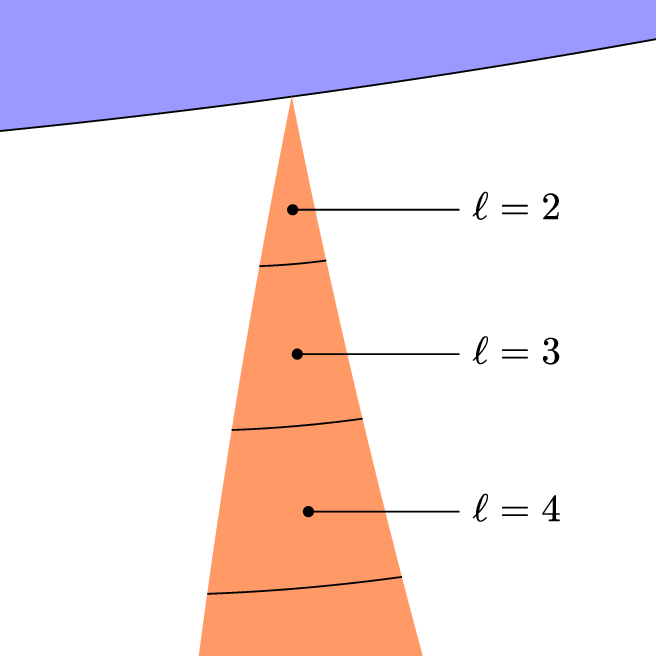}
\caption{
A theorised Arnold tongue of a piecewise-linear map that we show cannot arise.
The tongue is divided by curves of persistence-type border-collision bifurcations
where the number $\ell$ of points of the stable periodic solution located in $x_1 < 0$ changes by one.
\label{fig:nopePartitionedTongue}
} 
\end{center}
\end{figure}

\begin{figure}[t!]
\begin{center}
\includegraphics[width=8cm]{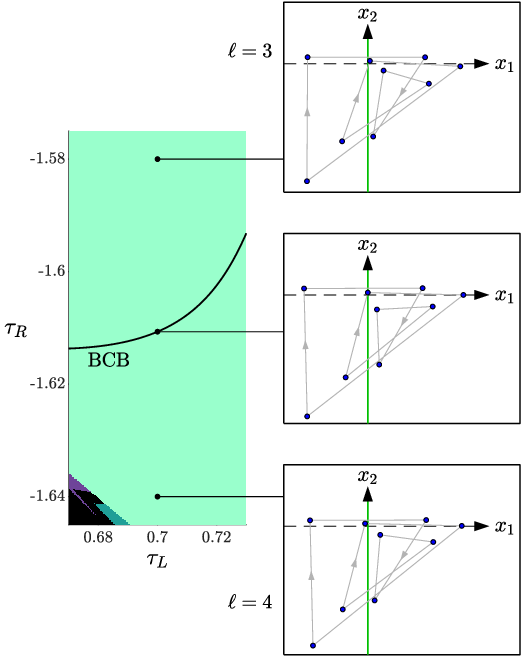}
\caption{
A two-parameter bifurcation diagram of the two-dimensional border-collision normal form
\eqref{eq:pwl} with \eqref{eq:exampleNonrotational} using $\delta_L = 0.1$ and $\delta_R = 1.2$.
This shows a curve of persistence-type BCBs (border-collision bifurcations)
where one point of a stable period-$9$ solution crosses the switching manifold, $x_1 = 0$.
(For parameter values in the lower-left the stable solution has higher period or is aperiodic.)
\label{fig:nopeTonguesSpecial10}
} 
\end{center}
\end{figure}

Heuristically the impossibility of Fig.~\ref{fig:nopePartitionedTongue} can explained with a simple argument.
The invariant circle in the Arnold tongue has stable and saddle periodic solutions
whose points alternate (stable, saddle, stable, etc) as we go around the circle.
Thus, assuming the circle intersects $x_1 = 0$ at exactly two points,
the stable and saddle solutions must have either the same number of points in $x_1 < 0$,
or their numbers of points in $x_1 < 0$ differ by one.
But they cannot have the same number of points in $x_1 < 0$
because they would have the same symbolic itinerary (i.e.~the same sequence of $L$'s and $R$'s)
and piecewise-linear maps only have one periodic solution with a given itinerary except at bifurcations.
Thus the numbers differ by one
and this must be maintained throughout the tongue.
So when one point of the stable solution crosses $x_1 = 0$,
one point of the unstable solution must also cross $x_1 = 0$,
and these two events cannot happen simultaneously in a codimension-one fashion.

The purpose of this paper we reach this conclusion in a rigorous manner and in a more general setting.
We show that when a piecewise-linear map \eqref{eq:pwl} is invertible and its fixed points do not suffer a degeneracy,
the restriction of the map to any attracting invariant circle is a degree-$1$ circle map (Theorem \ref{th:circleMapHasDegreeOne}).
We characterise the class of symbolic itineraries
that are possible for periodic solutions on degree-$1$ invariant circles
intersecting $x_1 = 0$ at exactly two points (Theorem \ref{th:perSolnIsRotational});
such periodic solutions can be associated to a rotation number and are termed {\em rotational}.
Finally we show that if a rotational periodic solution undergoes a BCB that does not alter its rotation number
then this bifurcation must be of nonsmooth-fold type whereby two solutions collide and annihilate (Theorem \ref{th:bcbIsNSFold}).
All three results are new; they are achieved by combining the ideas of Feigin \cite{Fe78}
for connecting stability to admissibility,
with the basic theory for periodic solutions of piecewise-linear maps,
and together with the fundamental principles of circle homeomorphisms.

It is worth stressing that the rotational symbolic itineraries
form a two-parameter family.
This reflects the two-dimensional sausage-string structure of the Arnold tongues
and is an extension of the one-parameter family of minimax itineraries
that are well known to characterise periodic solutions of degree-$1$ circle homeomorphisms in the classical setting \cite{GaLa84,GrAl17,HaZh98}.
In the classical setting the symbols correspond to different branches of the circle map;
in our setting the symbols correspond to different sides of $x_1 = 0$.

The remainder of this paper is organised as follows.
We start in \S\ref{sec:circleMaps} by brief reviewing circle homeomorphisms.
Next in \S\ref{sec:sym} we consider continuous maps on $\mathbb{R}^n$ with two smooth pieces
and formalise symbolic representations for their periodic solutions following \cite{Si17c}.
In this setting we state and prove Theorem \ref{th:perSolnIsRotational}
that periodic solutions on invariant circles have rotational symbolic itineraries.
In \S\ref{sec:perSolns} we return to the piecewise-linear form \eqref{eq:pwl}.
We show how periodic solutions can be computed explicitly and state and prove Theorem \ref{th:circleMapHasDegreeOne}.
Then in \S\ref{sec:bcbs} we consider BCBs of periodic solutions
and state and prove Theorem \ref{th:bcbIsNSFold}.
We then end in \S\ref{sec:conc} with a brief discussion.

\section{Circle maps}
\label{sec:circleMaps}

In this section we describe elementary concepts for maps on $\mathbb{S}^1$.
Further details can be found in the texts \cite{AlLl00,DeVa93,He79,KaHa95}.

We represent the circle $\mathbb{S}^1$ by the interval $[0,1)$, where $1$ is {\em identified} with $0$.
For any $w \in \mathbb{R}$ we let $\pi(w) = w - \lfloor w \rfloor$ be the fractional part of $w$.
So the function $\pi : \mathbb{R} \to \mathbb{S}^1$ is the natural projection of the real line onto the circle.

It is often helpful to extend circle maps onto $\mathbb{R}$ as follows.
Note that if $g : \mathbb{S}^1 \to \mathbb{S}^1$ is continuous then $g(1^-) = g(0)$,
where $g(1^-)$ is short-hand for $\lim_{t \to 1^-} g(t)$ (the limit from below).

\begin{definition}
A \dfem{lift} of a continuous map $g : \mathbb{S}^1 \to \mathbb{S}^1$
is a continuous function $G : \mathbb{R} \to \mathbb{R}$ with the property
that $\pi(G(w)) = g(\pi(w))$ for all $w \in \mathbb{R}$.
\label{df:lift}
\end{definition}

\begin{figure}[b!]
\begin{center}
\includegraphics[width=14cm]{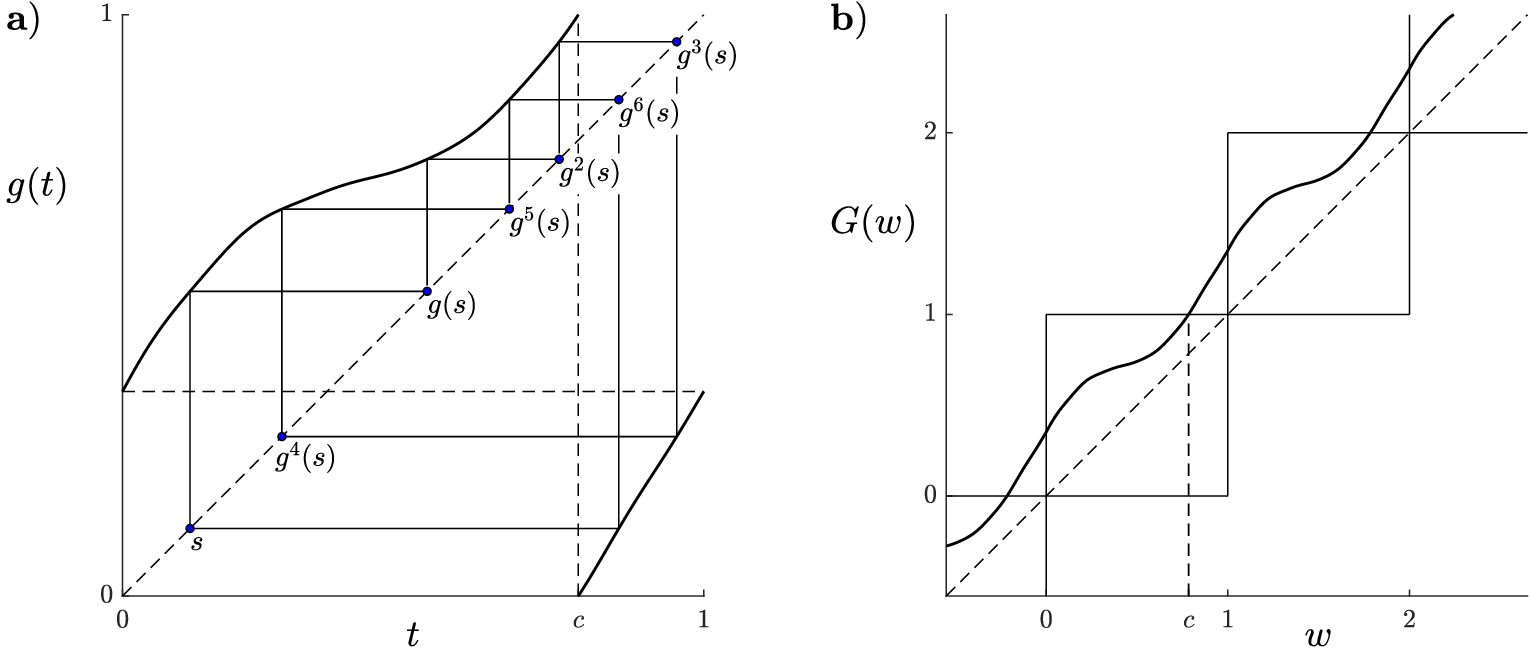}
\caption{
A continuous degree-$1$ circle map (panel (a)) and a corresponding lift (panel (b)).
The circle map has a stable periodic solution with rotation number $\rho = \frac{2}{7}$.
\label{fig:nopeMapAndLift}
} 
\end{center}
\end{figure}	

Fig.~\ref{fig:nopeMapAndLift} shows a circle map and a corresponding lift.
Lifts are unique up to translation by an integer
so the value of $G(1) - G(0)$ is independent of the lift.
This value is an integer because $g(1^-) = g(0)$ and called the degree of $g$:

\begin{definition}
The \dfem{degree} of a continuous map $g : \mathbb{S}^1 \to \mathbb{S}^1$
is the value of $G(1) - G(0)$, where $G$ is a lift of $g$.
\label{df:degree}
\end{definition}

If $g$ is a homeomorphism then its degree must be $1$ or $-1$.
If $g$ is a degree-$1$ homeomorphism
then for any lift $G$ there exists $r \in \mathbb{Z}$ such that
\begin{equation}
G(t) = \begin{cases}
g(t) + r, & t \in [0,c), \\
g(t) + r + 1, & t \in [c,1),
\end{cases}
\label{eq:G}
\end{equation}
where $c = g^{-1}(0)$ as in Fig.~\ref{fig:nopeMapAndLift}.

Next we define the rotation number.
This can be interpreted as the average amount by which iterates of $g$ step around the circle.
It is also the fraction of iterates that belong to $[c,1)$.

\begin{definition}
The \dfem{rotation number} of a degree-$1$ homeomorphism $g : \mathbb{S}^1 \to \mathbb{S}^1$ is
\begin{equation}
\rho = \lim_{i \to \infty} \pi \left( \frac{G^i(t) - t}{i} \right)
\label{eq:rotationNumber}
\end{equation}
for any lift $G$ of $g$ and $t \in [0,1)$.
\label{df:rotNum}
\end{definition}

As shown originally by Poincar\'e \cite{Po81},
the rotation number is well-defined because
the limit in \eqref{eq:rotationNumber} exists and is independent of $t$.

We conclude this section by characterising the way by which points in a periodic solution
of a degree-$1$ homeomorphism are ordered.
This uses the following definition.

\begin{definition}
Let $m$ and $p$ be positive coprime integers (i.e.~${\rm gcd}(m,p) = 1$).
The \dfem{multiplicative inverse} of $m$ modulo $p$ is the unique number $d \in \{ 1,2,\ldots,p-1 \}$
for which $m d ~{\rm mod}~ p = 1$.
\label{df:multInv}
\end{definition}

\begin{lemma}
Let $\xi$ be a period-$p$ solution of a degree-$1$ homeomorphism $g : \mathbb{S}^1 \to \mathbb{S}^1$.
Let $s \in [0,1)$ be the smallest point in $\xi$
and let $m$ be number of points in $\xi$ that belong to $[c,1)$.
Then $m$ and $p$ are coprime,
the rotation number of $g$ is $\frac{m}{p}$, and
\begin{equation}
s < g^d(s) < g^{2 d}(s) < \ldots < g^{(p-1)d}(s),
\label{eq:ordering}
\end{equation}
where $d$ is the multiplicative inverse of $m$ modulo $p$.
\label{le:ordering}
\end{lemma}

As an example consider the period $p=7$ solution in Fig.~\ref{fig:nopeMapAndLift}-a.
It has $m = 2$ points in $[c,1)$ and rotation number $\rho = \frac{2}{7}$.
Here $d = 4$ and indeed the points are ordered according to \eqref{eq:ordering}.
A proof of Lemma \ref{le:ordering} can be found in \cite{KaHa95};
for completeness a proof is provided in Appendix \ref{app:a}.

\section{Symbolic representations of periodic solutions}
\label{sec:sym}

Consider a map
\begin{equation}
f(x) = \begin{cases}
f_L(x), & x_1 \le 0, \\
f_R(x), & x_1 \ge 0,
\end{cases}
\label{eq:f}
\end{equation}
where $f_L$ and $f_R$ are each smooth and defined throughout $\mathbb{R}^n$.
The hyperplane $x_1 = 0$ is the {\em switching manifold} of $f$.
To describe periodic solutions of $f$ we use words comprised of $L$'s and $R$'s.
For example the upper phase portrait in Fig.~\ref{fig:nopeTonguesSpecial10}
shows an admissible stable $LLRRRRLRR$-cycle.
Formally we use the following definition.

\begin{definition}
Let $\Omega = \{ L, R \}$ and $p \ge 1$ be an integer.
An element of $\Omega^p$ is \dfem{word} $\cX$ of length $p$
that we write as $\cX = \cX_0 \cX_1 \cdots \cX_{p-1}$.
An \dfem{$\cX$-cycle} of a map \eqref{eq:f} is an ordered set $\{ y^{(0)}, y^{(1)}, \ldots, y^{(p-1)} \}$ of points in $\mathbb{R}^n$
for which
\begin{equation}
\begin{split}
f_{\cX_0}(y^{(0)}) &= y^{(1)}, \\
f_{\cX_1}(y^{(1)}) &= y^{(2)}, \\
&\hspace{2mm}\vdots \\
f_{\cX_{p-1}}(y^{(p-1)}) &= y^{(0)}.
\end{split}
\nonumber
\end{equation}
An $\cX$-cycle is \dfem{admissible}
if $y^{(i)}_1 < 0$ for all $i$ for which $\cX_i = L$,
and $y^{(i)}_1 > 0$ for all $i$ for which $\cX_i = R$.
An $\cX$-cycle is \dfem{virtual}
if either $y^{(i)}_1 > 0$ for some $i$ for which $\cX_i = L$,
or $y^{(i)}_1 < 0$ for some $i$ for which $\cX_i = R$.
\label{df:cycle}
\end{definition}

\begin{definition}
The \dfem{shift map} $\sigma : \Omega^p \to \Omega^p$ is defined by
\begin{equation}
\sigma(\cX) = \cX_1 \cX_2 \cdots \cX_{p-1} \cX_0 \,.
\label{eq:sigma}
\end{equation}
A word $\cY$ is said to be a \dfem{shift} of another word $\cZ$ if there exists an integer $i$ such that $\cY = \sigma^i(\cZ)$.
\label{df:shift}
\end{definition}

If a periodic solution of \eqref{eq:f} is an $\cX$-cycle,
then it is also a $\sigma^i(\cX)$-cycle for any integer $i$.
It follows that any periodic solution of \eqref{eq:f} with no points on the switching manifold
corresponds to a word that is unique up to a shift.

\begin{figure}[b!]
\begin{center}
\includegraphics[width=6cm]{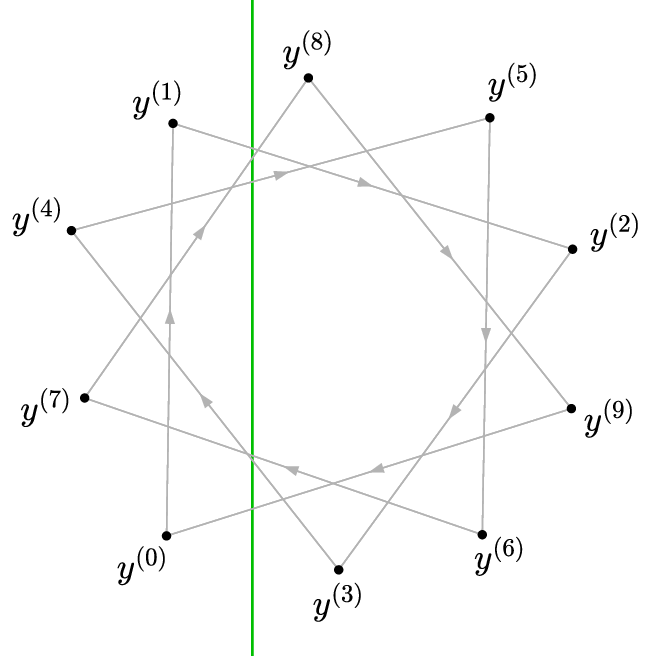}
\caption{
An $\cF[4,3,10]$-cycle.
The green line is the switching manifold $x_1 = 0$.
\label{fig:nopeRot}
} 
\end{center}
\end{figure}

To motivate the next definition, consider a period-$p$ solution of \eqref{eq:f}
whose points lie roughly on a circle, Fig.~\ref{fig:nopeRot}.
Suppose $\ell$ of these points lie to the left of the switching manifold
and let $y^{(0)}$ be the lower-most point located left of the switching manifold as indicated in the figure.
Let $y^{(i)} = f^i \left( y^{(0)} \right)$ for each $i = 1,2,\ldots,p-1$,
and suppose each point maps $m$ places clockwise.
Then each $y^{(i)}$ is located $i m ~{\rm mod}~ p$ places clockwise from $y^{(0)}$,
so $y^{(i)}$ lies to the left of the switching manifold if and only if $i m ~{\rm mod}~ p \in \{ 0,1,\ldots,\ell-1 \}$.

\begin{definition}
Let $\ell, m, p$ be positive integers with $\ell < p$, $m < p$, and $m$ and $p$ coprime.
Define the word $\cF[\ell,m,p] \in \Omega^p$ by
\begin{equation}
\cF_i = \begin{cases}
L, & i m ~{\rm mod}~ p < \ell, \\
R, & \text{otherwise}.
\end{cases}
\label{eq:cFi}
\end{equation}
A word $\cX$ is said to be \dfem{rotational} if it is a shift of some $\cF[\ell,m,p]$.
\label{df:cF}
\end{definition}

The following result provides an equivalent definition of $\cF[\ell,m,p]$.

\begin{lemma}
For any rotational word $\cF[\ell,m,p]$ and $j \in \{ 0,1,\ldots,p-1 \}$,
\begin{equation}
\cF_{j d \,{\rm mod}\, p} = \begin{cases}
L, & j < \ell, \\
R, & \text{otherwise},
\end{cases}
\label{eq:cFj}
\end{equation}
where $d$ is the multiplicative inverse of $m$ modulo $p$.
\label{le:cF1}
\end{lemma}

\begin{proof}
Notice $j < \ell$ if and only if $j d m ~{\rm mod}~ p < \ell$,
so by substituting $i = j d ~{\rm mod}~ p$ into \eqref{eq:cFi} we obtain \eqref{eq:cFj}.
\end{proof}

In the remainder of the paper we use the following notation.
Given $k \in \mathbb{Z}$ and a word $\cX \in \Omega^p$,
we write $\cX^{\overline{k}}$
for the word in $\Omega^p$ that is identical to $\cX$ except in its $i^{\rm th}$ symbol,
where $i = k ~{\rm mod}~ p$.
E.g.~if $\cX = LLLRR$ then $\cX^{\overline{2}} = LLRRR$.
The following lemmas will be useful in \S\ref{sec:bcbs}.

\begin{lemma}
Let $\cX = \cF[\ell,m,p]$ be rotational with $2 \le \ell \le p-2$
and $d$ be the multiplicative inverse of $m$ modulo $p$.
Given $j \in \{ 0,1,\ldots,p-1 \}$
the word $\cX^{\overline{j d}}$ is a shift of $\cF[\ell-1,m,p]$ or $\cF[\ell+1,m,p]$
if and only if $j \in \{ 0, \ell-1, \ell, p-1 \}$.
\label{le:cF2}
\end{lemma}

\begin{proof}
Throughout this proof we use \eqref{eq:cFj}.
Let $\cY = \cX^{\overline{j d}}$.
Suppose $j \in \{ 0,1,\ldots,\ell-1 \}$, equivalently $\cX_{j d \,{\rm mod}\, p} = L$
(the case $\cX_{j d \,{\rm mod}\, p} = R$ can be treated similarly).
If $j = \ell-1$ then $\cY = \cF[\ell-1,m,p]$,
and if $j = 0$ then $\sigma^d(\cY) = \cF[\ell-1,m,p]$.
Otherwise $\ell \ge 3$ and $\cY$ has the property
that $\cY_{(j-1) d \,{\rm mod}\, p} = L$, $\cY_{j d \,{\rm mod}\, p} = R$, and $\cY_{(j+1) d \,{\rm mod}\, p} = L$,
so cannot be a shift $\cF[\ell-1,m,p]$.
\end{proof}

\begin{lemma}
For any $\cX = \cF[\ell,m,p]$,
\begin{equation}
\sigma^{\ell d} \big( \cX^{\overline{0}} \big) = \sigma^{(\ell-1)d}(\cX)^{\overline{0}}.
\label{eq:id2}
\end{equation}
\label{le:cF3}
\end{lemma}

\begin{proof}
By \eqref{eq:cFj} $\cX_{j d \,{\rm mod}\, p} = L$ if and only if $j \in \{ 0, 1, \ldots, \ell-1 \}$.
Then $\cX^{\overline{0}}_{j d \,{\rm mod}\, p} = L$ if and only if $j \in \{ 1, 2, \ldots, \ell-1 \}$.
So the $j d^{\rm th}$ symbol of the left-hand side of \eqref{eq:id2} is $L$ 
if and only if $j + \ell \in \{ 1, 2, \ldots, \ell-1 \}$,
equivalently $j \in \{ p-\ell+1, p-\ell+2, \ldots, p-1 \}$.

The $j d^{\rm th}$ symbol of $\sigma^{(\ell-1)d}(\cX)$ is $L$
if and only if $j + \ell - 1 \in \{ 0, 1, \ldots, \ell-1 \}$,
equivalently $j \in \{ p-\ell+1, p-\ell+2, \ldots, p-1 \} \cup \{ 0 \}$.
Flipping the $0^{\rm th}$ symbol gives the right-hand side of \eqref{eq:id2},
and the above characterisation of the left-hand side of \eqref{eq:id2}.
\end{proof}

Finally in this section we consider periodic solutions on invariant circles of \eqref{eq:f} when this map is a homeomorphism.
To be clear, an {\em invariant circle} of $f$ is a simple closed curve $C \subset \mathbb{R}^n$ for which $f(C) = C$.
The restriction of $f$ to $C$ (denoted $f|_C$) is a circle map.

Suppose $C$ intersects the switching manifold at exactly two points and $f|_C$ is degree-$1$.
For any rotational word $\cX$ one can certainly engineer a piecewise-smooth map $f$ with such a circle
and on which the map has an admissible $\cX$-cycle.
The following result provides a converse
and shows that rotational words are {\em exactly} those
that correspond to periodic solutions on invariant circles with the given assumptions.
Given what we know about degree-$1$ circle homeomorphisms,
this result should not be surprising because by design $\cF[\ell,m,p]$
encodes rotation with rotation number $\frac{m}{p}$.

\begin{theorem}
Let $f$ be a homeomorphism of the form \eqref{eq:f}.
Suppose $f$ has an invariant circle $C$ intersecting the switching manifold $x_1 = 0$ at exactly two points
and $f|_C$ is degree-$1$.
Then for any admissible $\cX$-cycle on $C$ the word $\cX$ is rotational.
\label{th:perSolnIsRotational}
\end{theorem}

Fig.~\ref{fig:nopeInvCircleStable} provides an example.
Panel (a) is a phase portrait of the two-dimensional map \eqref{eq:exampleMapPWS}
with $(\alpha,\beta) = (-0.444,-0.6)$ that has an invariant circle $C$,
and panel (b) shows the restriction of the map to $C$.
Values $t$ in the domain of the circle map
correspond to points $x$ with $x_1 < 0$ if $t \in \left( 0, \frac{1}{2} \right)$ (orange),
and to points $x$ with $x_1 > 0$ if $y \in \left( \frac{1}{2}, 1 \right)$ (purple).
In the context of the circle map,
periodic points are assigned the symbol $L$ if they belong to $\left( 0, \frac{1}{2} \right)$
and assigned the symbol $R$ if they belong to $\left( \frac{1}{2}, 1 \right)$
This generalises the classical (smooth) situation
where symbols are instead used to distinguish the two branches of the circle map,
and we recover this scenario in the special case $c = \frac{1}{2}$.
In this case periodic solutions correspond to rotational words $\cF[\ell,m,p]$ with $\ell = p-m$.
The combinatorical properties of such words have been well studied, for instance they are {\em minimax} \cite{GaLa84}.

\begin{figure}[h!]
\begin{center}
\includegraphics[width=13cm]{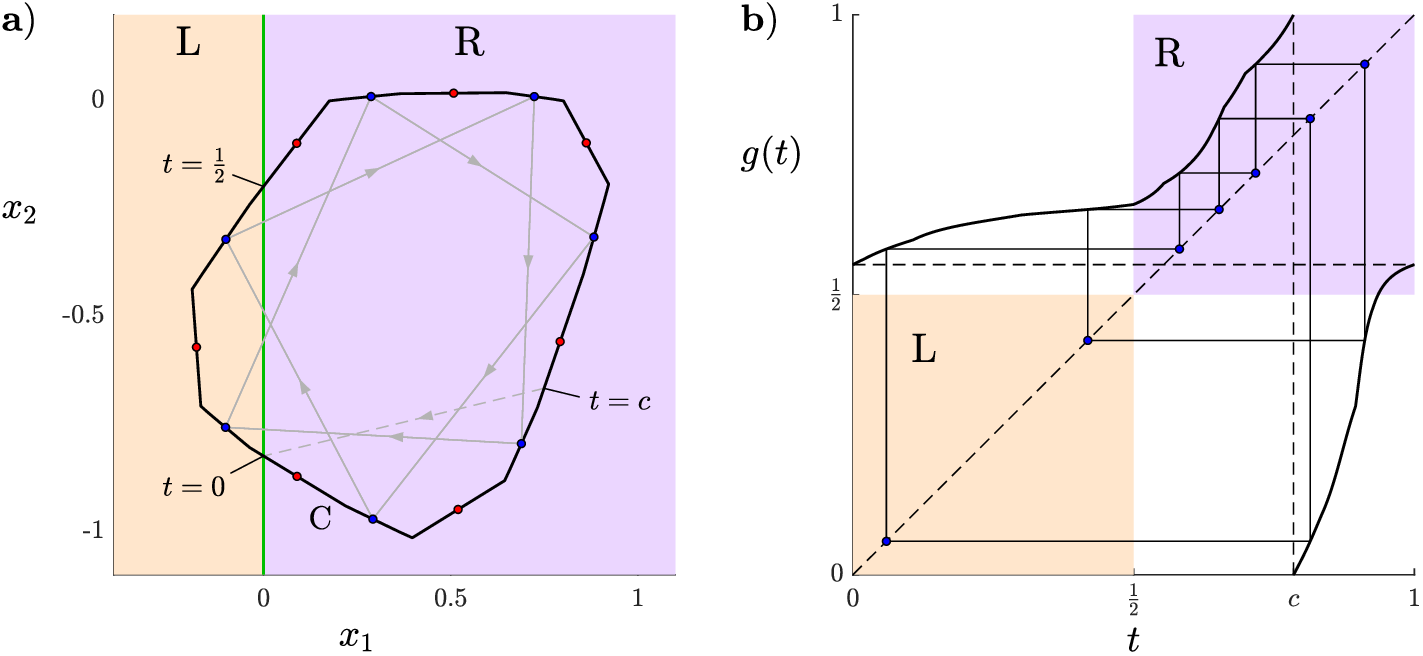}
\caption{
Panel (a) repeats Fig.~\ref{fig:nopePWLTongues}-a
showing a stable $\cF[2,2,7]$-cycle and an unstable $\cF[1,2,7]$-cycle on an invariant circle $C$.
Panel (b) shows the restriction of the map to $C$ and indicates the corresponding stable period-$7$ solution.
This map was constructed using $t = \varphi(x)$ where $\varphi : C \to [0,1)$
is defined by $t = \frac{\theta}{2 \pi}$ where $\theta$ is the angle
between the line segment from $x$ to $(0,-0.5)$ and the corresponding lower half of the switching manifold.
Points on $C$ that correspond to $t = 0$, $\frac{1}{2}$, and $c$ are indicated.
\label{fig:nopeInvCircleStable}
} 
\end{center}
\end{figure}

\begin{proof}[Proof of Theorem \ref{th:perSolnIsRotational}]
Let $\varphi : C \to [0,1)$ be a homeomorphism
with the $x_1 < 0$ part of $C$ mapping to $(0,\frac{1}{2})$
and the $x_1 > 0$ part of $C$ mapping to $(\frac{1}{2},1)$.
Then $f|_C$ is conjugate to the circle map $g = \varphi \circ f \circ \varphi^{-1}$.
Let $c = g^{-1}(0)$.

The image under $\varphi$ of the $\cX$-cycle is a period-$p$ solution of $g$.
Since $g$ is degree-$1$ this solution satisfies \eqref{eq:ordering}
where $s$ is the smallest point in the solution and $m$ is the number of points of the solution belonging to $[c,1)$.
Let $\ell$ be the number of points of the solution in $(0,\frac{1}{2})$.
Then $g^{j d}(s) \in (0,\frac{1}{2})$ if and only if $j \in \{ 0, 1, \ldots, \ell-1 \}$.

Write the $\cX$-cycle as $\{ y^{(0)}, y^{(1)}, \ldots, y^{(p-1)} \}$
where $y^{(i)} = \varphi^{-1} \left( g^i(s) \right)$ for each $i$.
By definition $\cX_i = L$ if and only if $y^{(i)}_1 < 0$.
Here $y^{(\ell d)}_1 < 0$ if and only if $g^{\ell d}(s) \in \{ 0, 1, \ldots, \ell-1 \}$,
so by \eqref{eq:cFj} we have $\cX = \cF[\ell,m,p]$.
\end{proof}

\section{Periodic solutions of piecewise-linear maps}
\label{sec:perSolns}

We continue to consider maps of the form \eqref{eq:f} but
now suppose $f_L$ is a linear function, i.e.
\begin{equation}
f_L(x) = A_L x + b,
\nonumber
\end{equation}
where $A_L \in \mathbb{R}^{n \times n}$ and $b \in \mathbb{R}^n$.
If $f_R$ is also linear then for $f$ to be continuous map we must have
\begin{equation}
f_R(x) = A_R x + b,
\nonumber
\end{equation}
where $A_R \in \mathbb{R}^{n \times n}$ differs from $A_L$ in only its first column.
This reproduces the piecewise-linear form \eqref{eq:pwl}.
Notice $f$ is a homeomorphism if and only if $\det(A_L) \det(A_R) > 0$.

Given a word $\cX \in \Omega^p$, let
\begin{equation}
f_\cX = f_{\cX_{p-1}} \circ \cdots \circ f_{\cX_1} \circ f_{\cX_0} \,,
\nonumber
\end{equation}
denote the composition of $f_L$ and $f_R$ in the order specified by $\cX$.
Since $f_L$ are $f_R$ are linear functions, so is $f_\cX$, specifically
\begin{equation}
f_\cX(x) = M_\cX x + P_\cX b,
\nonumber
\end{equation}
where
\begin{align}
M_\cX &= A_{\cX_{p-1}} \cdots A_{\cX_1} A_{\cX_0} \,, \label{eq:McW} \\
P_\cX &= A_{\cX_{p-1}} \cdots A_{\cX_2} A_{\cX_1} + \cdots + A_{\cX_{p-1}} A_{\cX_{p-2}} + A_{\cX_{p-1}} + I. \label{eq:PcW}
\end{align}
If $\det(I - M_{\cX}) \ne 0$ then $f_\cX$ has the unique fixed point
\begin{equation}
x^{\cX} = (I - M_{\cX})^{-1} P_{\cX} b.
\label{eq:xW}
\end{equation}
This point is one point of an $\cX$-cycle.
The remaining points are generated by iterating $x^{\cX}$
under $f_L$ and $f_R$ in the order specified by $\cX$.
In fact, since $f_\cX$ is a linear function,
a unique $\cX$-cycle exists if and only if $\det(I - M_{\cX}) \ne 0$.
But of course the $\cX$-cycle is not necessarily admissible.
This is governed by the signs of the first components of its points (Definition \ref{df:cycle}).

The Jacobian matrix of $f_\cX$ evaluated at $x^{\cX}$ (or evaluated at any point in $\mathbb{R}^n$) is $M_\cX$.
Thus the stability multipliers associated with an admissible $\cX$-cycle are the eigenvalues of $M_\cX$.
It follows that an admissible $\cX$-cycle of the piecewise-linear map $f$ is asymptotically stable if and only if
all eigenvalues of $M_\cX$ have modulus less than $1$.

Next we derive a concise formula for the first component of $x^\cX$.
This formula involves adjugate matrices:
if $A \in \mathbb{R}^{n \times n}$ is invertible the adjugate of $A$, denoted ${\rm adj}(A)$, satisfies
\begin{equation}
A^{-1} = \frac{{\rm adj}(A)}{\det(A)}.
\label{eq:Ainverse}
\end{equation}
More generally the adjugate is defined as follows \cite{Ko96,PiOd07}.

\begin{definition}
For all $i,j \in \{ 1,2,\ldots,n \}$
let $m_{ij}$ be the determinant of the $(n-1) \times (n-1)$ matrix
obtained by removing the $i^{\rm th}$ row and $j^{\rm th}$ column from $A$.
The \dfem{adjugate} of $A$ has $(i,j)$-entry $(-1)^{i+j} m_{ji}$ for all $i,j \in \{ 1,2,\ldots,n \}$.
\label{df:adjugate}
\end{definition}

Now consider the matrices $I - A_L$ and $I - A_R$.
These matrices differ only in their first columns,
thus by Definition \ref{df:adjugate} their adjugates have the same first row.
We denote this row $\varrho^{\sf T}$.

The significance of $\varrho^{\sf T}$ can be seen from the fixed points of $f$.
If $\det(I - A_L) \ne 0$ then $f_L$ has the unique fixed point $x^L = (I - A_L)^{-1} b$.
So by \eqref{eq:Ainverse} the first component of $x^L$ is
\begin{equation}
x^L_1 = \frac{\varrho^{\sf T} b}{\det(I - A_L)}.
\label{eq:xL1}
\end{equation}
Similarly if $\det(I - A_R) \ne 0$ the first component of the fixed point $x^R = (I - A_R)^{-1} b$ is
\begin{equation}
x^R_1 = \frac{\varrho^{\sf T} b}{\det(I - A_R)}.
\label{eq:xR1}
\end{equation}
The following result was first obtained in \cite{SiMe10} and will be needed in \S\ref{sec:bcbs}.
A proof is provided below for completeness.

\begin{proposition}
Let $\cX \in \Omega^p$ and suppose $\det(I - M_\cX) \ne 0$.
Then
\begin{equation}
x^{\cX}_1 = \frac{\det(P_\cX) \varrho^{\sf T} b}{\det(I - M_\cX)}.
\label{eq:xW1}
\end{equation}
\label{pr:xW1}
\end{proposition}

\begin{proof}
By \eqref{eq:McW} and \eqref{eq:PcW},
\begin{equation}
P_\cX \left( I - A_{\cX_0} \right)
= I - M_\cX + \sum_{j=1}^{p-1} A_{\cX_{p-1}} \cdots A_{\cX_{j+1}} \left( A_{\cX_j} - A_{\cX_0} \right).
\nonumber
\end{equation}
For each $j$ the matrices $A_{\cX_j}$ and $A_{\cX_0}$ are either equal or differ only in their first columns.
Thus $I - M_\cX$ and $P_\cX \left( I - A_{\cX_0} \right)$ differ only in their first columns.
Thus the first rows of the adjugates of $I - M_\cX$ and $P_\cX \left( I - A_{\cX_0} \right)$ are the same, i.e.
\begin{equation}
e_1^{\sf T} {\rm adj}(I - M_\cX) = e_1^{\sf T} {\rm adj} \left( P_\cX \left( I - A_{\cX_0} \right) \right),
\nonumber
\end{equation}
where $e_1$ is the first standard basis vector of $\mathbb{R}^n$.
By multiplying this equation by $P_\cX$ on the right we obtain
\begin{align}
e_1^{\sf T} {\rm adj}(I - M_\cX) P_\cX
&= e_1^{\sf T} {\rm adj} \left( I - A_{\cX_0} \right) {\rm adj}(P_\cX) P_\cX \nonumber \\
&= \varrho^{\sf T} \det(P_\cX), \nonumber
\end{align}
and the result follows from \eqref{eq:xW}.
\end{proof}

The next result shows that if an invariant circle $C$ is stable, meaning {\em Lyapunov stable} \cite{Ea98},
the restriction of $f$ to $C$ is a degree-$1$ circle map.
Fig.~\ref{fig:nopeInvCircleUnstable} provides an example to show that stability is necessary.
The proof of Theorem \ref{th:circleMapHasDegreeOne} below
uses a connection between the admissibility of the fixed points $x^L$ and $x^R$
and their stability multipliers.
This type of argument seems to have first been employed by Feigin \cite{Fe78}.

\begin{figure}[h!]
\begin{center}
\includegraphics[width=15.5cm]{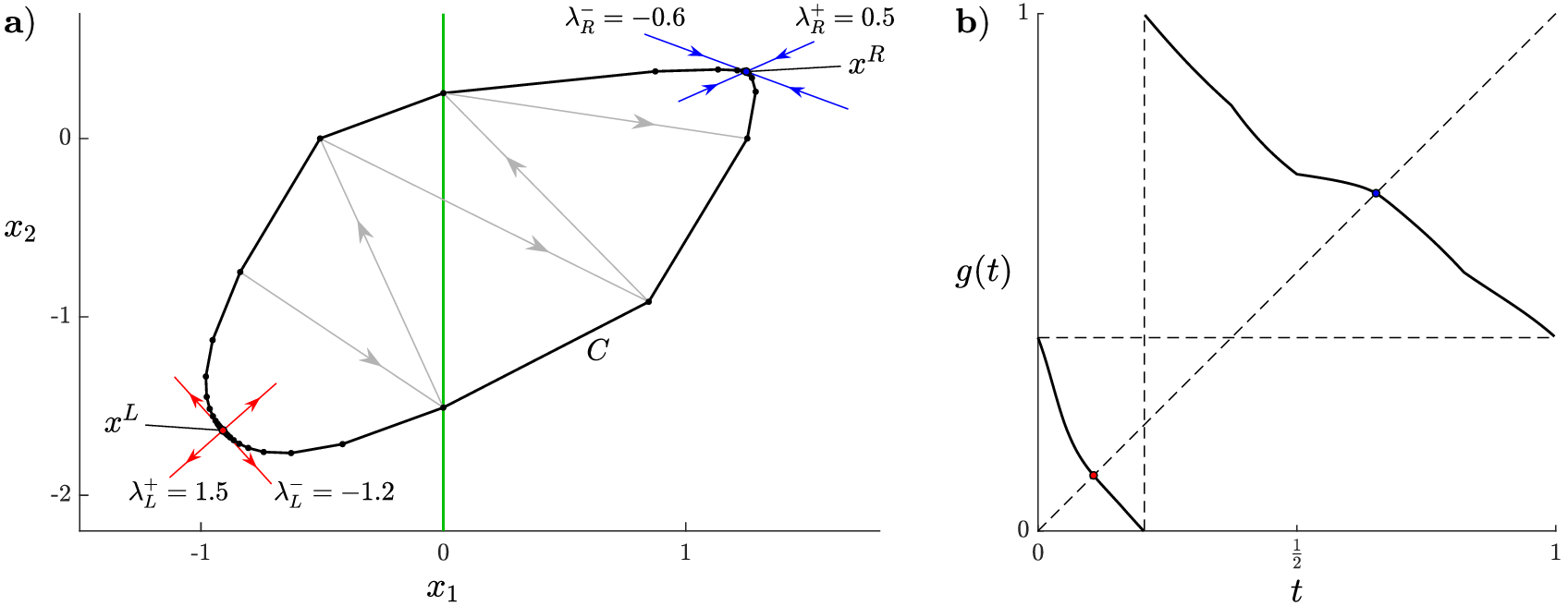}
\caption{
Panel (a) is a phase portrait of the two-dimensional border-collision normal form
\eqref{eq:pwl} with \eqref{eq:exampleNonrotational} and $(\tau_L,\delta_L,\tau_R,\delta_R) = (0.3,-1.8,-0.1,-0.3)$.
The map has an invariant circle $C$ 
and panel (b) shows that the restriction of the map to $C$ is a circle map $g$ with degree $-1$
(this map was constructed as in Fig.~\ref{fig:nopeInvCircleStable}).
Panel (a) also shows the fixed points $x^L$ and $x^R$, their stability multipliers,
and the associated eigendirections.
The circle $C$ is not stable because $x^L \in C$ is repelling.
\label{fig:nopeInvCircleUnstable}
} 
\end{center}
\end{figure}

\begin{theorem}
Let $f$ be a map \eqref{eq:pwl} with $\det(A_L) \det(A_R) > 0$,
$\det(I - A_L) \ne 0$, $\det(I - A_R) \ne 0$, and $\varrho^{\sf T} b \ne 0$,
where $\varrho^{\sf T}$ is the first row of ${\rm adj}(I - A_L)$.
If $f$ has a stable invariant circle $C$ then $f|_C$ is a degree-$1$ homeomorphism.
\label{th:circleMapHasDegreeOne}
\end{theorem}

\begin{proof}
The map $f$ is a homeomorphism because $\det(A_L) \det(A_R) > 0$.
Thus $f|_C$ is a homeomorphism so its degree is $1$ or $-1$.
Let $g$ be a circle map conjugate to $f|_C$ with domain $[0,1)$.
Suppose for a contradiction the degree of $g$ is $-1$.
Then $g$ has two fixed points and these must correspond to fixed points of $f$.
But $\det(I - A_L)$ and $\det(I - A_R)$ are nonzero
so the only possible fixed points of $f$ are $x^L$ and $x^R$ given above.
Thus $x^L$ and $x^R$ belong to $C$ and are not virtual.
Since $\varrho^{\sf T} b \ne 0$, by \eqref{eq:xL1} and \eqref{eq:xR1} this requires $\det(I - A_L)$ and $\det(I - A_R)$ have opposite signs.
Since the determinant of a matrix is the product of its eigenvalues,
at least one of $I - A_L$ and $I - A_R$ has a negative eigenvalue, suppose $I - A_L$, without loss of generality.
Then $A_L$ has an eigenvalue $\lambda > 1$.
So $x^L$ is unstable and $f$ has an orbit that emanates from $x^L$ on one side of $x^L$ (because the eigenvalue is positive).
This orbit cannot be contained within $C$, because $g$ is decreasing
so orbits near the corresponding fixed point of $g$ visit both sides of the fixed point,
hence the orbit must emanate from $C$.
But this contradicts the assumption that $C$ is stable,
so the degree of $g$ is $1$.
\end{proof}

\section{Border-collision bifurcation of rotational periodic solutions}
\label{sec:bcbs}

Now we consider one-parameter families of piecewise-linear maps \eqref{eq:pwl}.
That is, we suppose $A_L$, $A_R$, and $b$ vary continuously with a parameter $\eta \in \mathbb{R}$
(while maintaining $A_L$ and $A_R$ differing in only their first columns).

Suppose that at some $\eta = \eta^*$ an admissible $\cX$-cycle undergoes a
BCB whereby its $k^{\rm th}$ point 
collides with the switching manifold.
For genericity assume $\det(I - M_\cX) \ne 0$ at $\eta = \eta^*$.
If the bifurcation occurs in a generic fashion
then the $\cX$-cycle is admissible on one side of the bifurcation and virtual on the other side of the bifurcation.

Further suppose $\det \left( I - M_{\cX^{\overline{k}}} \right) \ne 0$ at $\eta = \eta^*$.
Then, locally, there exists a unique $\cX^{\overline{k}}$-cycle
that coincides with the $\cX$-cycle at $\eta = \eta^*$.
Again, generically, the $\cX^{\overline{k}}$-cycle will be
admissible on one side of the bifurcation and virtual on the other side of the bifurcation.

From these observations we can identify two distinct cases, as done originally by Brousin {\em et al.}~\cite{BrNe63}.
The $\cX$ and $\cX^{\overline{k}}$-cycles are either admissible on different sides of the bifurcation,
or admissible on the same side of the bifurcation.
The first case is termed {\em persistence} because, if we only consider admissible solutions,
as we pass through the bifurcation a single periodic solution appears to persist.
The second case is termed a {\em nonsmooth fold} because, with the same mindset,
two periodic solutions appear to collide and annihilate, akin to a saddle-node bifurcation, or fold.

Now, motivated by Theorem \ref{th:perSolnIsRotational} and our desire to understand periodic solutions in Arnold tongues,
we suppose $\cX$ and $\cX^{\overline{k}}$ are rotational and correspond to the same rotation number $\frac{m}{p}$.
Lemma \ref{le:cF2} tells us what values of $k$ need to be considered.
The following result shows that the BCB must be a nonsmooth fold.
Notice it makes no assumptions on the stability of the periodic solutions or the existence of an invariant circle.

\begin{theorem}
Suppose the BCB at $\eta = \eta^*$ is generic in the sense that
$\det(I - M_\cX)$,
$\det \left( I - M_{\cX^{\overline{k}}} \right)$,
and $\varrho^{\sf T} b$ are nonzero at $\eta = \eta^*$,
and that the $\cX$ and $\cX^{\overline{k}}$-cycles
are each admissible on exactly one side of the bifurcation.
If $\cX = \cF[\ell,m,p]$ and $k = j d$, where $j \in \{ 0, \ell-1, \ell, p-1 \}$,
then the BCB at $\eta = \eta^*$ is a nonsmooth fold.
\label{th:bcbIsNSFold}
\end{theorem}

\begin{proof}
By symmetry it suffices to consider $j = 0$.
Let $\{ y^{(0)}, y^{(1)}, \ldots, y^{(p-1)} \}$ denote the $\cX$-cycle
and $\{ z^{(0)}, z^{(1)}, \ldots, z^{(p-1)} \}$ denote the $\cX^{\overline{0}}$-cycle.
These vary continuously with $\eta$ in a neighbourhood of $\eta^*$,
and at $\eta = \eta^*$ we have $y^{(i)} = z^{(i)}$ for each $i$.
At $\eta = \eta^*$ we also have $y^{(i)}_1 = z^{(i)}_1 = 0$.

By \eqref{eq:xW1} the first components of $y^{(0)}$ and $z^{(0)}$ are
\begin{align}
y^{(0)}_1 &= \frac{\det(P_\cX) \varrho^{\sf T} b}{\det(I - M_\cX)}, \nonumber \\
z^{(0)}_1 &= \frac{\det(P_\cX) \varrho^{\sf T} b}{\det(I - M_{\cX^{\overline{0}}})}, \nonumber
\end{align}
where in the second equation we have used the fact that $P_{\cX^{\overline{0}}} = P_\cX$ by \eqref{eq:PcW}.
Thus, by genericity, the value of $\det(P_\cX)$ changes sign at $\eta = \eta^*$.
Hence the bifurcation corresponds to persistence if $\det(I - M_\cX)$ and $\det(I - M_{\cX^{\overline{0}}})$
have the same sign and is a nonsmooth fold if they have different signs.

Equation \eqref{eq:xW1} also gives
\begin{align}
y^{((\ell - 1)d)}_1 &= \frac{\det \left( P_{\sigma^{(\ell-1)d}(\cX)} \right) \varrho^{\sf T} b}{\det(I - M_\cX)}, \label{eq:y9} \\
z^{(\ell d)}_1 &= \frac{\det \left( P_{\sigma^{\ell d} \left( \cX^{\overline{0}} \right)} \right) \varrho^{\sf T} b}{\det \left( I - M_{\cX^{\overline{0}}} \right)}. \label{eq:y19}
\end{align}
In the denominator of \eqref{eq:y9}
we have used the fact that the determinant of $I - M_{\sigma^i(\cX)}$ is independent of $i$,
and similarly for the denominator of \eqref{eq:y19}.
To \eqref{eq:y19} we apply \eqref{eq:id2} giving
\begin{equation}
z^{(\ell d)}_1 = \frac{\det \left( P_{\sigma^{(\ell-1)d}(\cX)} \right) \varrho^{\sf T} b}
{\det \left( I - M_{\cX^{\overline{0}}} \right)}, \label{eq:z9}
\end{equation}
where we have again used \eqref{eq:PcW} to drop the $\overline{0}$.
But $\cX_{(\ell - 1)d} = L$, so $y^{((\ell-1)d)}_1 < 0$ by the admissibility assumption,
and $\cX^{\overline{0}}_{\ell d} = R$, so $z^{(\ell d)}_1 > 0$ also by admissibility.
Thus by \eqref{eq:y9} and \eqref{eq:z9}
the signs of $\det(I - M_\cX)$ and $\det(I - M_{\cX^{\overline{0}}})$ are different, hence the BCB is a nonsmooth fold.
\end{proof}

\section{Discussion}
\label{sec:conc}

The dynamics of two-piece, piecewise-linear maps can be incredibly complex \cite{Si23c},
but of course there are limits to this complexity, as they have an exceedingly simple form,
and in this paper we have found new ways to bring these limits to light.
We have utilised the fact that such maps
cannot have more than one asymptotically stable fixed point
to show that stable invariant circles have degree $1$.
By combining this with the classical theory of circle homeomorphisms we have
provided new insights into the sausage-string structure.
The results show that two parts of an Arnold tongue
corresponding to consecutive values of $\ell$ cannot be joined by a codimension-one curve of persistence-type BCBs.

It remains to see if similar ideas can advance
our understanding of other bifurcation structures
that arise in families of piecewise-smooth maps.
For one-dimensional piecewise-smooth maps the structures are in the large well understood
through renormalisation, asymptotic calculations, and other means \cite{GrAl17,AvGa19},
but in higher dimensions several new structures have recently been described \cite{CaPa23,GaRa23,Si20}
and that remain to be fully understood.

\section*{Acknowledgements}

This work was supported by Marsden Fund contract MAU2209 managed by Royal Society Te Ap\={a}rangi.

\appendix
\section{Proof of Lemma \ref{le:ordering}}
\label{app:a}

Let $G$ be a lift of $g$ and let $r \in \mathbb{Z}$ be as in \eqref{eq:G}.
Since $\xi$ is period-$p$ with $m$ points in $[c,1)$,
\begin{equation}
G^p(s) = s + r p + m,
\label{eq:orderingProof1}
\end{equation}
so \eqref{eq:rotationNumber} with $t = s$ gives $\rho = \frac{m}{p}$.

We now define a sequence $S = \left\{ w^{(i)} \right\}_{i \in \mathbb{Z}}$ of points in $\mathbb{R}$ as follows.
First let $w^{(0)},\ldots,w^{(p-1)}$ be the points of $\xi$ in ascending order, i.e.
\begin{equation}
s = w^{(0)} < w^{(1)} < \cdots < w^{(p-1)}.
\nonumber
\end{equation}
Then for each $q \in \{ 0,1,\ldots,p-1 \}$ and $k \in \mathbb{Z}$ let $w^{(q + k p)} = w^{(q)} + k$.
Notice $S$ is an increasing sequence,
i.e.~$w^{(i)} < w^{(j)}$ if and only if $i < j$.

For all $i \in \mathbb{Z}$ there exists $a_i \in \mathbb{Z}$ such that $G(w^{(i)}) = w^{(i + a_i)}$.
We will now show that $a_i$ is independent of $i$.
Choose any $i, j \in \mathbb{Z}$ with $i < j$.
The number of points in $S$ between $w^{(i)}$ and $w^{(j)}$ is $j-i-1$.
Since $G$ is increasing this is the same as the number of points in $S$ between
$G(w^{(i)}) = w^{(i + a_i)}$ and $G(w^{(j)}) = w^{(j + a_j)}$.
But the number of points in $S$ between
$w^{(i + a_i)}$ and $w^{(j + a_j)}$ is $j + a_j - (i + a_i) - 1$,
and matching the two numbers gives $a_i = a_j$.	

Thus $a_i$ is independent of $i$, call it $a$.
Notice $G(s) = G(w^{(0)}) = w^{(a)}$,
so $G^i(s) = w^{(i a)}$ for all $i \in \mathbb{Z}$,
and in particular $G^p(s) = w^{(p a)} = s + a$.
By matching this to \eqref{eq:orderingProof1} we obtain $a = r p + m$.
Thus
\begin{equation}
g^i(s) = w^{(i a \,{\rm mod}\, p)} = w^{(i m \,{\rm mod}\, p)},
\nonumber
\end{equation}
for each $i \in \{ 0,1,\ldots,p-1 \}$.
From this we observe $m$ and $p$ are coprime
because $\xi$ consists of $p$ distinct points.
Also
\begin{equation}
g^{j d}(s) = w^{(j d m \,{\rm mod}\, p)} = w^{(j)},
\nonumber
\end{equation}
for each $j \in \{ 0,1,\ldots,p-1 \}$ giving \eqref{eq:ordering} as required.
\hfill $\Box$


\end{document}